\documentclass[a4paper,12pt]{amsart}
\usepackage{amsmath,amssymb,color,setspace}
\usepackage{fullpage,graphicx}
\usepackage[pdftitle = {Simplicial arrangements with special vertex}]{hyperref}

\newtheorem{Lemma}{Lemma}[section]
\newtheorem{Theorem}[Lemma]{Theorem}
\newtheorem{Proposition}[Lemma]{Proposition}

\theoremstyle{definition}
\newtheorem{Definition}[Lemma]{Definition}
\newtheorem{Remark}[Lemma]{Remark}
\newtheorem{Example}[Lemma]{Example}
\numberwithin{equation}{section}

\DeclareMathOperator{\SL}{SL}
\DeclareMathOperator{\vol}{vol}

\definecolor{dblue}{rgb}{0.0,0.1,0.6}
\newcommand{\df}[1]{{\bf\color{dblue} #1}}

\title{Simplicial arrangements with special vertex}

\author{Michael~Cuntz}
\address{Michael Cuntz, Leibniz Universit\"at Hannover,
Institut f\"ur Algebra, Zahlentheorie und Dis\-krete Mathematik,
Fakult\"at f\"ur Mathematik und Physik,
Welfengarten 1,
D-30167 Hannover, Germany}
\email{cuntz@math.uni-hannover.de}
\urladdr{https://www.iazd.uni-hannover.de/de/cuntz}


\begin{document}

\begin{abstract}
Shi-Catalan and affine Weyl arrangements are obtained in a similar way: they consist of a Weyl arrangement embedded into a higher dimensional space together with shifted copies. We consider the more general case when the Weyl arrangement is replaced by an arbitrary finite arrangement of hyperplanes. Such an arrangement is called an arrangement with special vertex. As our main result we prove that the number of parallel classes of lines in certain simplicial arrangements with special vertex is at most $12$.
\end{abstract}

\maketitle

\section{Introduction}

Simplicial arrangements are arrangements of hyperplanes in a finite dimensional real vector space such that every chamber is an open simplicial cone \cite{M41}. They turned out to be the right context to solve Brieskorn's conjecture \cite{D72}, but so far they are only conjectured to be all known \cite{G09}. Real reflection arrangements are the most important special cases.

One can extend the notion of simpliciality from finite arrangements to the infinite case (cf.\ \cite{CMW17}) in which the open simplicial cones are contained in a cone called the Tits cone. This includes the reflection arrangements for affine Weyl groups, for which it was recently shown that their complexified complements also define $K(\pi,1)$-spaces \cite{PS21}, thus extending Deligne's result.

A class of simplicial arrangements which is completely classified is the class of crystallographic arrangements \cite{C10} corresponding to Weyl groupoids \cite{CH10}. A first result on a classification of affine crystallographic arrangements is \cite{C14}, where we consider the special case when all lines are parallel to lines through one particular point and such that the lines through this point have the structure of a Weyl groupoid of rank two (or equivalently a Conway-Coxeter frieze pattern).

In this note we consider (possibly infinite) arrangements of hyperplanes which consist of a finite linear arrangement of hyperplanes called the special vertex together with parallel copies of these hyperplanes (see Definition \ref{def:spvt}), like for example the arrangements in Figure \ref{G2:1315}.
\begin{figure}[ht]
\includegraphics[width=0.4\textwidth]{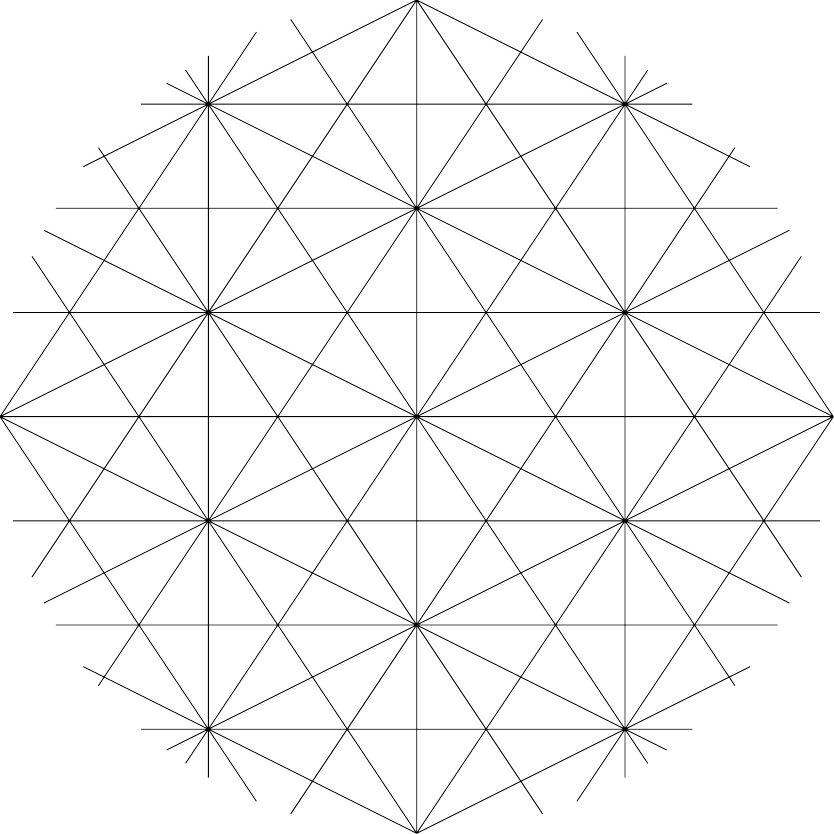}
\includegraphics[width=0.4\textwidth]{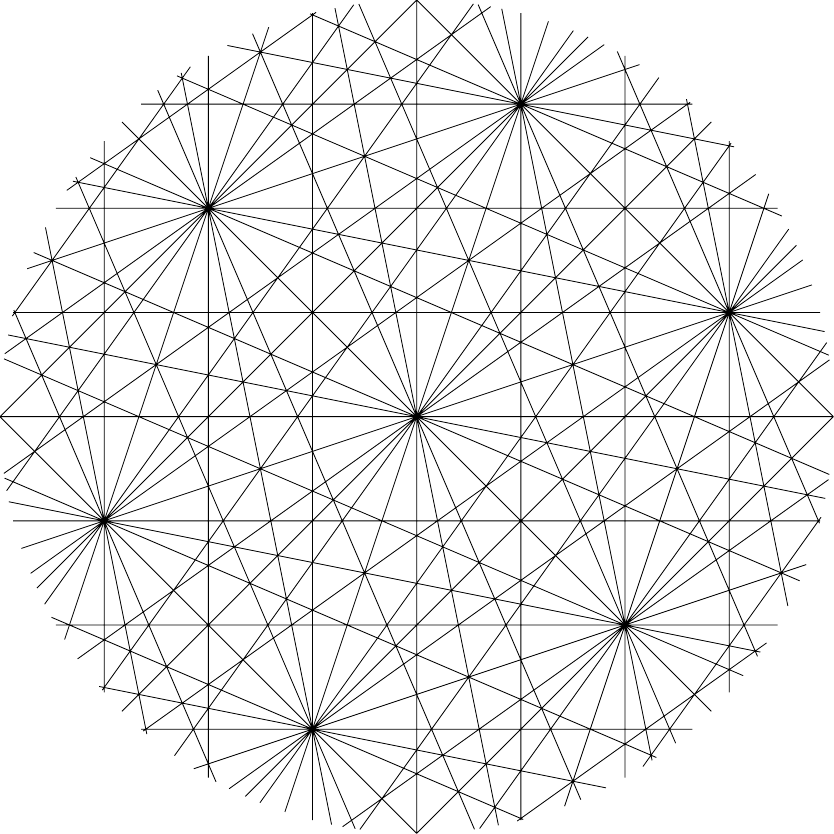}
\caption{\label{G2:1315}Toric $G_2$ and $(1,3,1,5,1,3,1,5,1,3,1,5)$ arrangement, cf.\ \cite{C14}}
\end{figure}

In the special case of rank three and when the shifts in all directions are parametrized by $\mathbb{Z}$, one can recover the arrangement via its Pl\"ucker matrix (Definition \ref{def:plueck}).
This matrix can also be viewed as a frieze pattern with coefficients as introduced in \cite{CHJ20}.

We focus on simplicial arrangements with special vertex. First we observe:
\begin{Theorem}[Theorem \ref{thm:1}]
Let $v_1,\ldots,v_n\in\mathbb{Z}^2$, $\mathcal{A}$ be the corresponding toric arrangement with special vertex and $(c_{i,j})_{i,j}$ the Pl\"ucker matrix and
for $1\le i_1,\ldots,i_k \le n$ let
$d_{i_1,\ldots,i_k}:=\vol(v_{i_1},\ldots,v_{i_k}):=\gcd\{ c_{i,j} \mid i,j\in\{i_1,\ldots,i_k\}\}$.
Then $\mathcal{A}$ is simplicial if and only if
\begin{equation}
\sum_{k\ge 2} \sum_{1\le i_1<\cdots<i_k\le n} (-1)^{k+1}(2k-3) d_{i_1,\ldots,i_k} = 0.
\end{equation}
\end{Theorem}

In the last section, we prove our main result which gives a bound on the number of parallel classes of lines in a simplicial arrangement with special vertex:

\begin{Theorem}[Theorem \ref{thm:2}]
Let $v_1,\ldots,v_n\in\mathbb{Z}^2$, $\mathcal{A}$ be the corresponding toric arrangement with special vertex.
Write $v_i=(a_i,b_i)$ and assume that $\gcd(a_i,b_i)=1$ for all $i=1,\ldots,n$.
If $\mathcal{A}$ is simplicial, then $n\le 12$.
\end{Theorem}

One of the main theorems of \cite{C14} can be obtained via this theorem since it implies that there are only finitely many cases left to consider. We thus obtain a completely different proof of a more general result.

\section{Arrangements with special vertex}

The following example motivates our main definition:

\begin{Example}
Let $v_1=(1,0,0), v_2=(1,1,0), v_3=(0,1,0)\in \mathbb{R}^3$, $e_1,e_2,e_3$ be the standard basis, and denote $v^\perp$ the hyperplane given by the orthogonal complement of $v\in \mathbb{R}^3$.
The arrangement
$$ \mathcal{A} := \{ (v_i+ k e_3)^\perp \mid i=1,\ldots,3, \:\: k \in \mathbb{Z} \}\cup \{ e_3^\perp \} $$
is the reflection arrangement of the affine Weyl group of Dynkin type $A_2$ (Figure \ref{fig:affshi} on the left).
Intersecting each hyperplane of $\mathcal{A}$ with the affine hyperplane $e_3^\perp+e_3$ gives an arrangement
$\mathcal{G}$ of lines in $\mathbb{R}^2$.
There is a fundamental domain in $\mathbb{R}^2$ such that the lines of $\mathcal{G}$ repeat periodically, thus $\mathcal{G}$ induces a toric arrangement on the quotient.
Moreover, every line is parallel to a line through the origin; the origin is therefore called a \df{special vertex} (cf.\ \cite[1.3.7]{BT72}).
\end{Example}

This concept can be generalized to the following class of arrangements:

\begin{Definition}\label{def:spvt}
Let $V=\mathbb{R}^\ell$, $v_1,\ldots,v_n\in V$, and $S_1,\ldots,S_n\subseteq \mathbb{Z}$ such that
$0\in S_i$ for $i=1,\ldots,n$. Denote $e_1,\ldots,e_{\ell+1}$ the standard basis of $\mathbb{R}^{\ell+1}$ and identify $V$ with $\langle e_1,\ldots,e_\ell\rangle$.
We say that
$$ \mathcal{A} := \{ (v_i+ k e_{\ell+1})^\perp \mid i=1,\ldots,n, \:\: k \in S_i \}\cup \{ e_{\ell+1}^\perp \} $$
is an \df{arrangement with special vertex of rank $\ell+1$}.
\end{Definition}

\begin{figure}[ht]
\reflectbox{\includegraphics[width=0.3\textwidth]{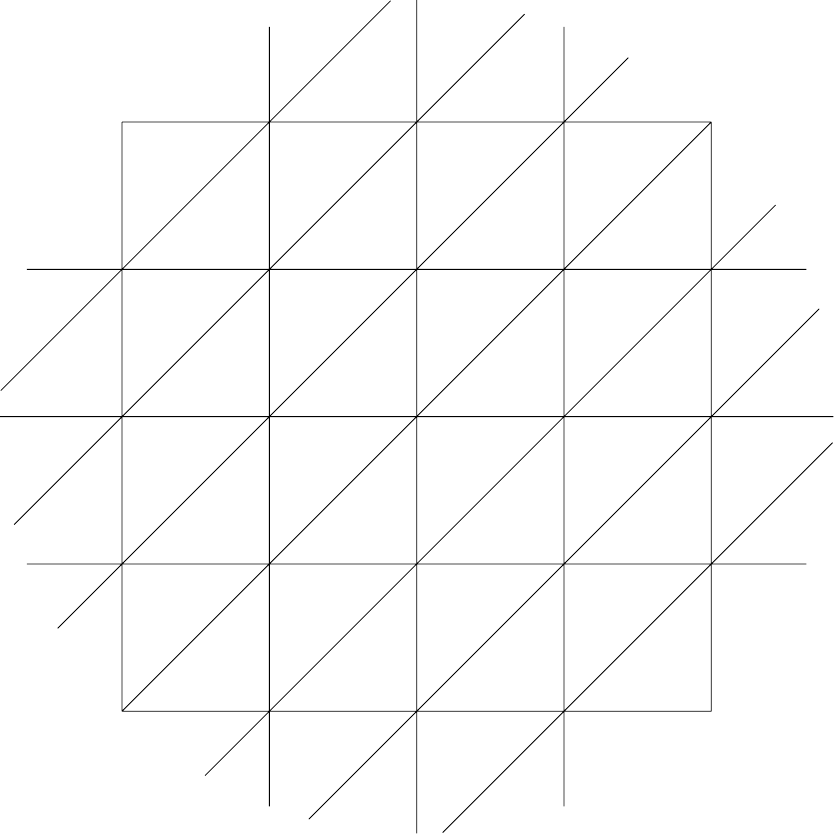}}
\reflectbox{\includegraphics[width=0.3\textwidth]{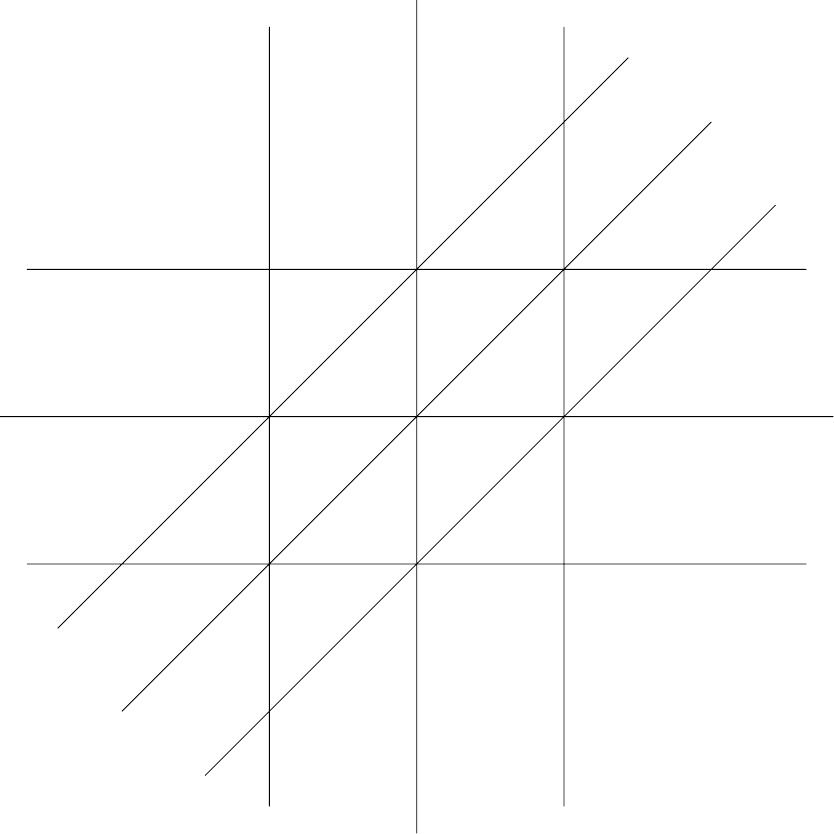}}
\caption{Affine and Shi-Catalan arrangements}
\label{fig:affshi}
\end{figure}

\begin{Example}
Let $v_1:=e_1$, $v_2:=e_2$, $v_3:=e_1+e_2$.
The arrangement
$$ \mathcal{A} := \{ (v_i+ k e_3)^\perp \mid i=1,\ldots,3, \:\: k \in \{-1,0,1\} \}\cup \{ e_3^\perp \} $$
is an example of a \df{Shi-Catalan arrangement} (Figure \ref{fig:affshi} on the right).
It is an arrangement with special vertex.
\end{Example}

In the case when all the sets $S_i$ are equal to $\mathbb{Z}$, we obtain \df{fundamental domains}:

\begin{Proposition}\label{prop:decofund}
Let $\mathcal{A}$ be an arrangement with special vertex for $v_1,\ldots,v_n\in\mathbb{Q}^\ell$ such that $\langle v_1,\ldots,v_n\rangle_{\mathbb{Q}}=\mathbb{Q}^\ell$ and
$S_i=\mathbb{Z}$ for all $i=1,\ldots,n$, let
$H_0:=e_{\ell+1}^\perp+e_{\ell+1}$ and
$$
\mathcal{G} := \{H\cap H_0 \mid H\in \mathcal{A}\}
$$
be the affine arrangement obtained by intersecting the hyperplanes of $\mathcal{A}$ with $H_0$.
Then there exists a basis $u_1,\ldots,u_\ell \in \mathbb{Q}^\ell$ such that
$$
\forall H\in \mathcal{G},\:\: w \in \sum_{i=1}^\ell \mathbb{Z}u_i\:\: :\:\:
H+w \in \mathcal{G}.
$$
\end{Proposition}

\begin{proof}
Let $\tilde u_1,\ldots,\tilde u_\ell$ be a basis of $\mathbb{Q}^\ell$ and $r_i\in \mathbb{Z}$ be the least common multiple of the denominators of
$$ (v_j,\tilde u_i), \quad j=1,\ldots,\ell. $$
Let $u_i:=r_i \tilde u_i$, $i=1,\ldots,\ell$.
Then $(v_j,u_i)\in \mathbb{Z}$ for all $i,j$.

For $H\in\mathcal{G}$, there exists a $v_i$ and $k\in\mathbb{Z}$ such that $H = (v_i+ke_{\ell+1})^\perp \cap H_0$, i.e., the points on $H$ are of the form $x+e_{\ell+1}$, $x\in V$ and $(v_i,x)=-k$.
Then for any $w\in \sum_{j=1}^\ell \mathbb{Z}u_j$,
$$ (v_i,x+w) \in \mathbb{Z}, $$
hence $x+w+e_{\ell+1}\in (v_i+k' e_{\ell+1})^\perp \cap H_0$
for $k':=-(v_i,x+w)$.
\end{proof}

\begin{Definition}
If $S_i=\mathbb{Z}$ for all $i=1,\ldots,n$ and $v_1,\ldots,v_n\in\mathbb{Q}^\ell$, then the corresponding arrangement $\mathcal{A}$ is a \df{toric arrangement with special vertex}.
The arrangement $\mathcal{G}$ defined in Proposition \ref{prop:decofund} is called the corresponding \df{affine arrangement}.
\end{Definition}

We use the notion of an affine simplicial arrangement; this is a special Tits arrangement whose Tits cone is a half space (cf.\ \cite{CMW17}):

\begin{Definition}
Let $\mathcal{A}$ be a toric arrangement with special vertex.
We say that $\mathcal{A}$
is \df{simplicial} if the connected components of the complement of $\mathcal{A}$ in
$$\{ (x_1,\ldots,x_{\ell+1})\in\mathbb{R}^{\ell+1} \mid x_{\ell+1}>0 \}$$
are open simplicial cones.
\end{Definition}

For example if $\ell=2$, simpliciality of a toric arrangement translates to the property that the corresponding affine arrangement triangulates the plane.

\begin{Example}
The arrangements in Figure \ref{G2:1315} are simplicial. The arrangement on the left is the set of reflecting hyperplanes of the affine Weyl group of type $G_2$. The arrangement on the right is a simplicial arrangement corresponding to the quiddity cycle $(1,3,1,5,1,3,1,5,1,3,1,5)$ as defined in \cite{C14}.
\end{Example}

\section{Pl\"ucker coordinates}

We now focus on the case $\ell+1=3$.
Our goal is to understand the structure of toric arrangements with special vertex. It turns out that the essential information is encoded in the matrix of Pl\"ucker coordinates:

\begin{Definition}\label{def:plueck}
Let $v_1,\ldots,v_n\in\mathbb{R}^2$. The matrix $(c_{i,j})_{1\le i,j\le n}$,
$$
c_{i,j} := \det(v_i v_j)
$$
is called the \df{Pl\"ucker matrix} for $v_1,\ldots,v_n$.
If $\mathcal{A}$ is an arrangement with special vertex of rank $3$ for $v_1,\ldots,v_n$ then we call the Pl\"ucker matrix of $v_1,\ldots,v_n$ the \df{Pl\"ucker matrix} of $\mathcal{A}$.
\end{Definition}

\begin{Remark}
If $c_{i,i+1}\ne 0$ for all $i$, then
the Pl\"ucker matrix ``is'' a tame \df{frieze pattern with coefficients} $\mathcal{F}$ as introduced in \cite[Def.\ 2.1]{CHJ20}, i.e., the map
$$
\mathcal{F} \::\: \{(i,j) \mid 1\le i\le j \le n\} \rightarrow \mathbb{R},
\quad (i,j) \mapsto c_{i,j}
$$
satisfies the \df{Ptolemy relations} (or Pl\"ucker relations)
\begin{equation}\label{eq:ptol}
c_{i,k} c_{j,\ell} = c_{i,\ell} c_{j,k} + c_{i,j} c_{k,\ell}
\end{equation}
for all $1\le i\le j\le k\le \ell\le n$.
(These hold for $c_{i,j}$ with $i\le j$, but can be extended to arbitrary $c_{i,j}$
via $c_{i,j}=-c_{j,i}$.) Note that this implies
\begin{equation}\label{eq:ptol2}
c_{j,k} = \frac{c_{i,j} c_{i+1,k} - c_{i,k} c_{i+1,j}}{c_{i,i+1}}
= \frac{1}{c_{i,i+1}} \det\begin{pmatrix}
c_{i,j} & c_{i,k} \\
c_{i+1,j} & c_{i+1,k}
\end{pmatrix}
\end{equation}
for all $i,j,k$ with $c_{i,i+1}\ne 0$.
\end{Remark}

\begin{Remark}\label{rem:sl2}
(i) By \cite[Thm.\ 3.7]{C23} in the special case $k=2$, given a matrix $(c_{i,j})_{i,j}$ with $c_{i,j}\in\mathbb{Z}\setminus\{0\}$ for $i\ne j$ satisfying the Ptolemy relations,
there exist $v_1,\ldots,v_n\in\mathbb{Z}^2$ such that
$c_{i,j}=\det(v_i v_j)$ for $1\le i,j\le n$.

\noindent
(ii) On the other hand, let $v_1,\ldots,v_n\in\mathbb{Z}^2$ be given with Pl\"ucker matrix $(c_{i,j})_{i,j}$ and let $A\in\SL_2(\mathbb{Z})$. Then $Av_1,\ldots,Av_n$ yield the same Pl\"ucker matrix $(c_{i,j})_{i,j}$.
\end{Remark}

Now assume that $v_1,\ldots,v_n\in\mathbb{Z}^2$ and $S_i=\mathbb{Z}$ for all $i=1,\ldots,n$. Then we can count the intersection points of $\mathcal{G}$ in a fundamental domain:

\begin{Lemma}\label{lem:mult}
Let $v_1,\ldots,v_n\in\mathbb{Z}^2$ and $\mathcal{G}$ be the corresponding affine arrangement.
Then $u_1=(1,0),u_2=(0,1)$ determine a fundamental domain
$$
D := \{ (x,y)\in \mathbb{R}^2 \mid 0\le x,y <1\}.
$$
A pair of classes of parallel lines of $\mathcal{G}$ corresponding to $v_i$ and $v_j$ intersect exactly
$|c_{i,j}|=|\det(v_i v_j)|$ times in $D$.
\end{Lemma}
\begin{proof}
Write $v_i=(a,b)$ and let $(x,y)$ be a point on a parallel to $v_i^\perp$.
Then $(a,b,k)\cdot (x,y,1)=0$ for some $k\in\mathbb{Z}$.
If $c,d\in \mathbb{Z}$, then
$$
0=(a,b,k+m)\cdot (x+c,y+d,1)=ax+by+k+ac+bd+m\in \mathbb{Z}
$$
for $m\in\mathbb{Z}$, thus $(x+c,y+d,1)$ also lies on some parallel to $v_i^\perp$.
Hence $D$ is a fundamental domain.
The second assertion comes from the fact that $|c_{i,j}|$ is the volume of the parallelotope spanned by $v_i$ and $v_j$ (cf.\ Lemma \ref{lem:inter}).
\end{proof}

\begin{Lemma}($\gcd$-condition)\label{lem:gcd}
Let $v_1,\ldots,v_n\in\mathbb{Z}^2$ and $(c_{ij})_{i,j}$ be the Pl\"ucker matrix.
Write $v_i=(a_i,b_i)$ and assume that $\gcd(a_i,b_i)=1$ for all $i=1,\ldots,n$. Then for all $1\le i<j<k\le n$:
\begin{equation}\label{gcd3}
\gcd(c_{i,j},c_{i,k}) = \gcd(c_{i,j},c_{j,k}) = \gcd(c_{i,k},c_{j,k}).
\end{equation}
\end{Lemma}
\begin{proof}
Up to a base change in $\SL_2(\mathbb{Z})$ (Remark \ref{rem:sl2}), $v_i$ may be chosen to be $(1,0)$. Then $c_{i,j}=b_j$, $c_{i,k}=b_k$, $c_{j,k}=a_jb_k-a_kb_j$.
Let $g:=\gcd(c_{i,j},c_{i,k})=\gcd(b_j,b_k)$. Then $g$ divides $c_{j,k}$, thus $g$ divides $\gcd(c_{i,j},c_{j,k})$ and $\gcd(c_{i,k},c_{j,k})$ as well.
The same argument with $j$ and $k$ instead of $i$ gives the claim.
\end{proof}

\begin{Definition}\label{def:gcd}
We say that $v_1,\ldots,v_n\in\mathbb{Z}^2$ satisfy the \df{$\gcd$-condition}
if the Pl\"ucker matrix satisfies Equation \eqref{gcd3} for all $1\le i<j<k \le n$.
\end{Definition}

\section{Simpliciality via Pl\"ucker coordinates}

In this section,
let $v_1,\ldots,v_n\in\mathbb{Z}^2$, $\mathcal{A}$ resp.\ $\mathcal{G}$ be the corresponding toric resp.\ affine arrangement with special vertex and $(c_{i,j})_{i,j}$ be the Pl\"ucker matrix.

For $k\ge 2$ and $1\le i_1<\ldots<i_k \le n$ let
\begin{equation}\label{eq:di}
d_{i_1,\ldots,i_k}:=\vol(v_{i_1},\ldots,v_{i_k}):=\gcd\{ c_{i,j} \mid i,j\in\{i_1,\ldots,i_k\}\}.
\end{equation}
\begin{Lemma}\label{lem:inter}
The number of intersection points of lines of the classes $i_1,\ldots,i_k$ in the fundamental domain $D$ is exactly
$d_{i_1,\ldots,i_k}$.
\end{Lemma}
\begin{proof}
An intersection point in $D$ has rational coordinates, thus it is of the form $(r,s)$ for $r,s\in\mathbb{Q}$ with $0\le r,s<1$, and it satisfies
$(v_{i_\nu},(r,s)) \in \mathbb{Z}$
for $\nu=1,\ldots,k$.
With
$$ x:=((v_{i_1})_1,\ldots,(v_{i_k})_1),\quad y:=((v_{i_1})_2,\ldots,(v_{i_k})_2),$$
this condition translates to
$rx+sy\in \mathbb{Z}^k$. We are thus counting the number of integral points in the parallelotope spanned by $x$ and $y$, which
is $\vol(x,y)=\vol(v_{i_1},\ldots,v_{i_k})$.
\end{proof}

\begin{Theorem}\label{thm:1}
Let $d_{i_1,\ldots,i_k}$ be defined as in Equation (\ref{eq:di}).
Then $\mathcal{A}$ is simplicial if and only if
\begin{equation}\label{eq:simp}
\sum_{k\ge 2} \sum_{1\le i_1<\cdots<i_k\le n} (-1)^{k+1}(2k-3) d_{i_1,\ldots,i_k} = 0.
\end{equation}
\end{Theorem}
\begin{proof}
Since the corresponding affine arrangement $\mathcal{G}$ is toric, it is a simplicial line arrangement if and only if the intersection points in the fundamental domain have multiplicity $3$ in average. 
We show with Lemma \ref{lem:inter} that this is equivalent to Equation \ref{eq:simp}:

Let $\nu_k$ be the number of intersection points of multiplicity $k$ in the fundamental domain. Then $\mathcal{A}$ is simplicial if and only if
$$
\sum_{k\ge 2} (k-3) \nu_k = 0.
$$
By Lemma \ref{lem:inter},
$$
\sum_{1\le i_1<\cdots<i_k\le n} d_{i_1,\ldots,i_k} = \sum_{s\ge k} \binom{s}{k} \nu_s.
$$
Hence
\begin{eqnarray*}
\sum_{k\ge 2} (-1)^{k+1} (2k-3)\sum_{1\le i_1<\cdots<i_k\le n} d_{i_1,\ldots,i_k} &=&
\sum_{k\ge 2} (-1)^{k+1} (2k-3) \sum_{s\ge k} \binom{s}{k} \nu_s\\
&=&\sum_{s\ge 2} \left(\sum_{2\le k\le s} (-1)^{k+1} (2k-3) \binom{s}{k}\right) \nu_s
\\
&=& \sum_{s\ge 2} (s-3) \nu_s
\end{eqnarray*}
since $\sum_{k=0}^s (-1)^{k+1} k \binom{s}{k} = 0$
and $\sum_{k=0}^s (-1)^k \binom{s}{k} = 0$.
\end{proof}

\section{Maximal simplicial arrangements with special vertex}

The statement of the following Lemma is illustrated in Figure \ref{G2p1p2} for the example of the affine Weyl group of type $G_2$.

\begin{figure}
    \centering
    \includegraphics[width=0.4\linewidth]{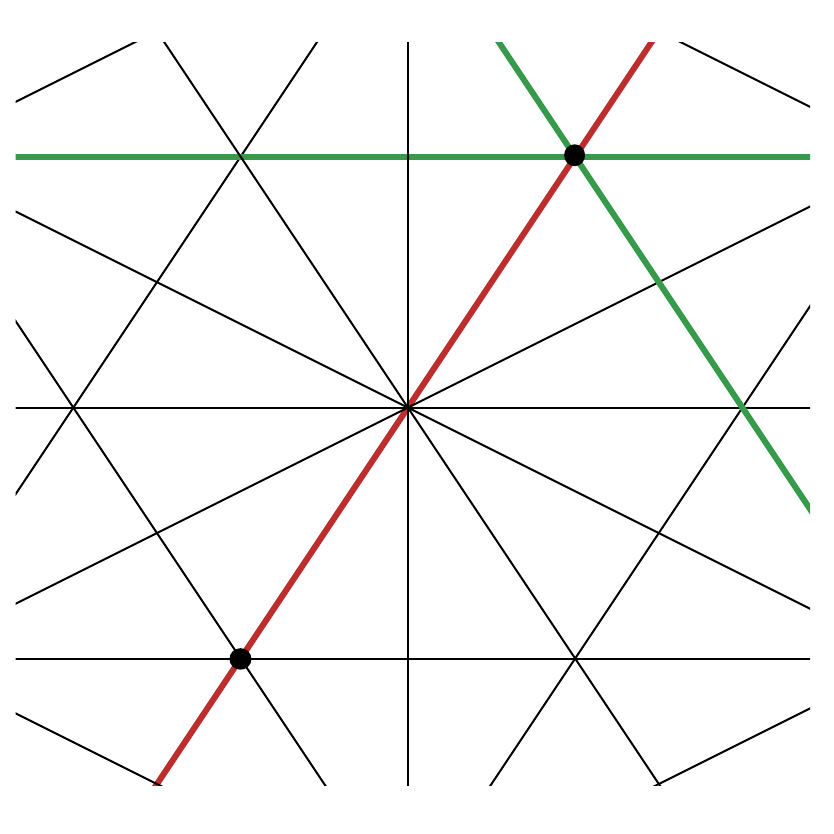}
    \caption{Points $p_1,p_2$ on the red line as in Lemma \ref{lem:close}}
    \label{G2p1p2}
\end{figure}

\begin{Lemma}\label{lem:close}
For each line $v_i^\perp$ through the origin, there are two intersection points closest to the origin, say $p_1,p_2$. Then $p_1=-p_2$ and
$$
\{ H\in \mathcal{G} \mid p_1 \in H \text{ or } p_2 \in H\} =
\{v_i^\perp\} \cup \{ v_j^\perp \pm e_3 \mid j=1,\ldots,n,\:\: |c_{i,j}|=m \}
$$
where $m:=\max\{ |c_{i,k}|\mid k=1,\ldots,n\}$.
\end{Lemma}
\begin{proof}
Up to a base change in $\SL_2(\mathbb{Z})$, without loss of generality $v_i=(d,0)$ for some $0<d\in\mathbb{Z}$.
By Lemma \ref{lem:mult}, parallels to $v_i^\perp$ and $v_j^\perp$ intersect $|c_{i,j}|$ times in the fundamental domain. Thus parallels to $v_j^\perp$ intersect $v_i^\perp$ exactly $|c_{i,j}|/d$ in the fundamental domain. So the closest intersection points are those for which $|c_{i,j}|$ is maximal.
\end{proof}

\begin{Definition}
We consider the relation $\le$ on $\mathbb{Z}^2$ with
$$ (a,b)\le (c,d) \quad :\Longleftrightarrow\quad ad\le cb.$$
A finite set $\{v_1,\ldots,v_n\}$ is totally ordered by $\le$ if $\det(v_iv_j)\ne 0$ for all $i\ne j$.
\end{Definition}

Viewing $v_1,\ldots,v_n$ as lines through the special vertex, this sorts the lines in a way which is compatible with chambers at the vertex.

\begin{Lemma}\label{lem:clvt}
Let $v_1>\ldots>v_n$ and write $c_{i,j+n}:=-c_{i,j}=c_{j,i}$ for $1\le i,j\le n$.
For $i=1,\ldots,n$, let
$$
m_i:=\max\{ |c_{i,k}|\mid k=1,\ldots,n\}, \quad
M_i:=\{j\in\{i+1,\ldots,i+n-1\} \mid |c_{i,j}|=m_i\}.
$$
Assume that all the chambers around the origin are simplicial cones.
Then there exist $j_1,\ldots,j_n$ with $i\le j_i\le i+n$ for $i=1,\ldots,n$ such that
$$
M_i\subseteq \{j_i,\ldots,j_{i+1}\}
$$
for all $i=1,\ldots,n$ where $j_{n+1}:=j_1+n$.
Moreover, $j_i,j_{i+1}\in M_i$ for all $i=1,\ldots,n$.
\end{Lemma}
\begin{proof}
For each $i$, $M_i$ consists of the indices of the lines meeting line $i$ at a point closest to the origin (Lemma \ref{lem:close}). As $v_j$ are ordered, the slopes of these lines are increasing with the indices. Thus the last index $j$ in $M_i$ labels the line which closes the chamber between line $i$ and $i+1$; by the same argument $j$ is also the first index in $M_{i+1}$. Therefore $j$ is the greatest element of $M_i$ and the smallest element of $M_{i+1}$. With $j_{i+1}:=j$ we obtain the intervals as claimed.
\end{proof}

The key for our main result is to consider intersection points behind the double points:

\begin{Lemma}\label{lem:clvt2}
Let $v_1>\ldots>v_n\in\mathbb{Z}^2$, $\mathcal{A}$ be the corresponding toric arrangement with special vertex and $(c_{i,j})_{i,j}$ be the Pl\"ucker matrix.
Write $c_{i,j+n}:=-c_{i,j}=c_{j,i}$ for $1\le i,j\le n$.
For $i=1,\ldots,n$, let
\begin{eqnarray*}
\mu_i&:=&\max\{ |c_{i,k}|\mid k=1,\ldots,n, \:\: |\{j\in\{1,\ldots,n\} \mid |c_{i,j}|=|c_{i,k}|\}|>1 \}, \\
N_i&:=&\{j\in\{i+1,\ldots,i+n-1\} \mid |c_{i,j}|=\mu_i\}.
\end{eqnarray*}
If $\mathcal{A}$ is simplicial,
then there exist $\sigma_i,\tau_i\in\mathbb{Z}$ for $i=1,\ldots,n$ such that
$$
N_i\subseteq \{\sigma_i,\ldots,\tau_i\}
\quad
\text{and}
\quad
[\sigma_i,\tau_i]\cap [\sigma_{i+4},\tau_{i+4}]\le 1
$$
for all $i=1,\ldots,n$,
where for $i>n$, $\sigma_i:=\sigma_{i-n}$ and $\tau_i:=\tau_{i-n}$.
Moreover, if $|c_{i,k}|>\mu_i$ for some $k$, then $k\in \{\sigma_i,\ldots,\tau_i\}$.
\end{Lemma}

\begin{figure}
\centering
\begin{picture}(160,160)
\put(0,0){\includegraphics[width=0.4\textwidth]{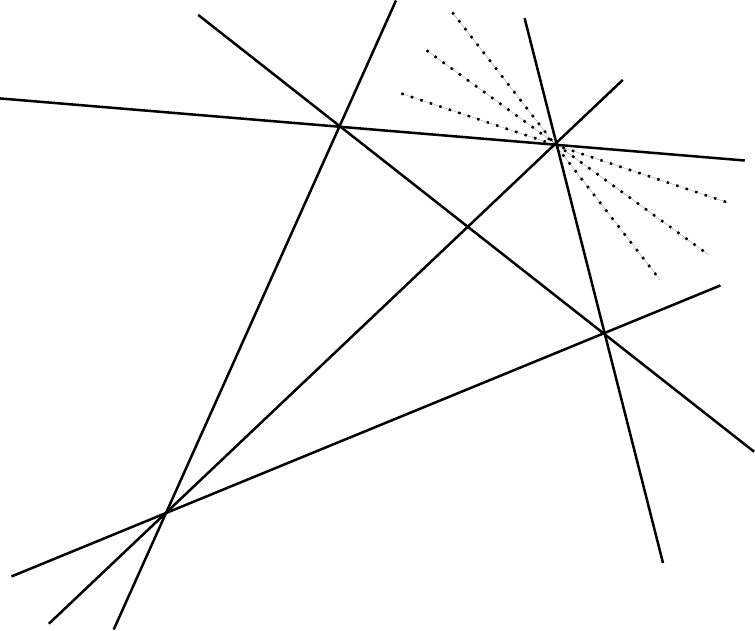}}
\put(85,80){$i$}
\put(30,113){$\tau_i$}
\put(138,30){$\sigma_i$}
\end{picture}
\hspace{30pt}
\begin{picture}(160,160)
\put(0,0){\includegraphics[width=0.4\textwidth]{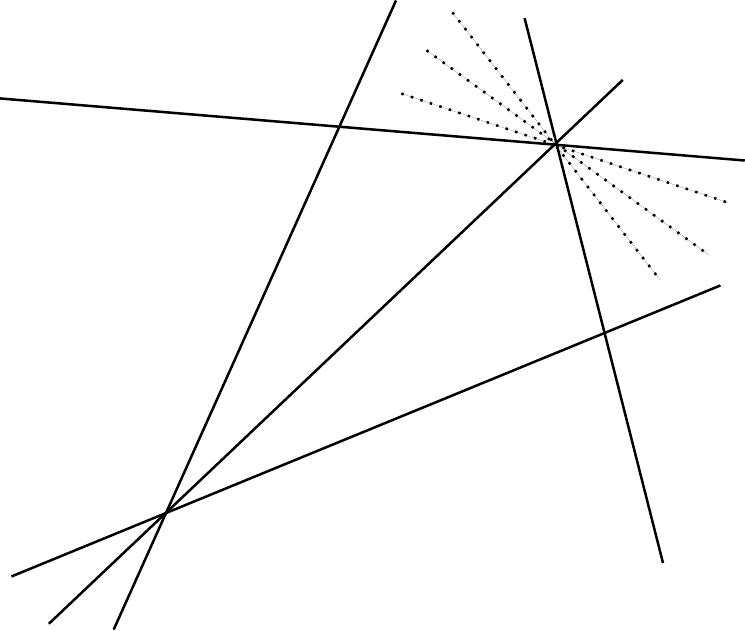}}
\put(85,80){$i$}
\put(30,113){$\tau_i$}
\put(138,30){$\sigma_i$}
\end{picture}
\caption{Proof of Lemma \ref{lem:clvt2}: case (1), case (2)}
\label{fig:prlem2}
\end{figure}
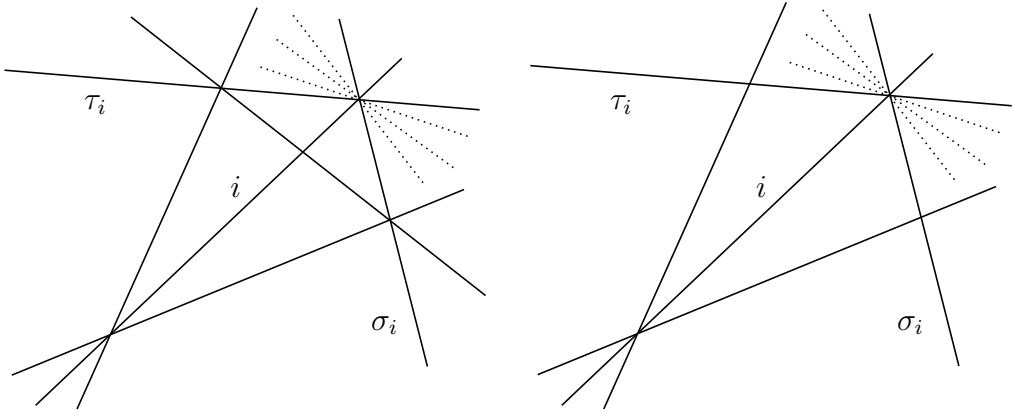

\begin{proof}
For each $i$, $N_i$ consists of the indices of the lines meeting line $i$ at a point closest to the origin that is not a double point (Lemma \ref{lem:close}). As $v_j$ are ordered, the slopes of these lines are increasing with the indices.
Thus there are two cases (Figure \ref{fig:prlem2}):\\
(1) The point closest to the origin on line $i$ is a double point $P$:
The point $P$ is the intersection point of line $i$ and a parallel to line $j$ for some index $j$.
There are four chambers around $P$, two of them are adjacent to the origin. The other two are beyond the parallel to $j$ and have walls $i,j$ and $k_1$ respectively $k_2$. Thus $k_1$ and $k_2$ are incident to the closest point to $i-1$ respectively $i+1$.
We set $\sigma_i:=k_1$, $\tau_i:=k_2$ and obtain
$N_i\subseteq \{\sigma_i,\ldots,\tau_i\}$.
\\
(2) The point closest to the origin on line $i$ is not a double point:
The two chambers adjacent to line $i$ and the origin have walls $(i-1,k_1,i)$ and $(i,k_2,i+1)$ for some $k_1<k_2$.
Again, $k_1$ and $k_2$ are incident to the closest point to $i-1$ respectively $i+1$, we set $\sigma_i:=k_1$, $\tau_i:=k_2$. and obtain $N_i\subseteq \{\sigma_i,\ldots,\tau_i\}$.
\smallskip\\
The definitions of $\sigma_i,\tau_i$ yield
$$
\sigma_i \in [\sigma_{i-1},\tau_{i-1}], \quad
\tau_i \in [\sigma_{i+1},\tau_{i+1}]
$$
for all $i=1,\ldots,n$, and $\tau_i=\sigma_{i+1}$ if $i$ and $i+1$ are both in case (2).
\\
Moreover, for $i\ne j$, $[\sigma_i,\tau_i]$ and $[\sigma_j,\tau_j]$ can have at most one common element if the points closest to the origin at lines $i$ and $j$ are not double points.
Since $\mathcal{A}$ is simplicial, there are no consecutive double points.
Thus even if every second point around the origin is a double point, $[\sigma_i,\tau_i]$ and $[\sigma_{i+4},\tau_{i+4}]$ can have at most $1$ common element. (The distance $4$ is only necessary when $i,i+2$, and $i+4$ have double points.)
\end{proof}

\begin{Example}
The quiddity cycle $(1,3,1,5,1,3,1,5,1,3,1,5)$ defines a Weyl groupoid of rank two which has the set
$$
R_+ := \{
(1, 0), (1, 1), (2, 3), (1, 2), (3, 7), (2, 5), (3, 8), (1, 3), (2, 7), (1, 4), (1, 5), (0, 1)\}
$$
as a set of positive roots.
With $R_+=:\{v_1,\ldots,v_{12}\}$ one obtains the following Pl\"ucker matrix:
$$
\begin{array}{ccccccccccccccccccc}
&&&&&\ddots &&&&&&&&&&&&& \\
0 & 1 & 1 & 2 & 1 & \fcolorbox{blue}{white}{3} & 2 & \fcolorbox{blue}{white}{3} & 1 & 2 & 1 & 1 & 0 &  &  &  &  &  & \\
 & 0 & 1 & 3 & 2 & \fcolorbox{green}{white}{7} & 5 & \fcolorbox{blue}{white}{8} & 3 & \fcolorbox{green}{white}{7} & 4 & 5 & 1 & 0 &  &  &  &  & \\
 &  & 0 & 1 & 1 & 4 & 3 & \fcolorbox{blue}{white}{5} & 2 & \fcolorbox{blue}{white}{5} & 3 & 4 & 1 & 1 & 0 &  &  &  & \\
 &  &  & 0 & 1 & 5 & 4 & \fcolorbox{green}{white}{7} & 3 & \fcolorbox{blue}{white}{8} & 5 & \fcolorbox{green}{white}{7} & 2 & 3 & 1 & 0 &  &  & \\
 &  &  &  & 0 & 1 & 1 & 2 & 1 & \fcolorbox{blue}{white}{3} & 2 & \fcolorbox{blue}{white}{3} & 1 & 2 & 1 & 1 & 0 &  & \\
 &  &  &  &  & 0 & 1 & 3 & 2 & \fcolorbox{green}{white}{7} & 5 & \fcolorbox{blue}{white}{8} & 3 & \fcolorbox{green}{white}{7} & 4 & 5 & 1 & 0 & \\
  &  &  &  &  &  & 0 & 1 & 1 & 4 & 3 & \fcolorbox{blue}{white}{5} & 2 & \fcolorbox{blue}{white}{5} & 3 & 4 & 1 & 1 & 0 \\
&&&&&&&&&&&&&\ddots &&&&&
\end{array}$$
The entries in the blue boxes mark the indices $j_1,\ldots,j_n$ as in Lemma \ref{lem:clvt}; the green boxes correspond to the pairs $\sigma_i,\tau_i$ from Lemma \ref{lem:clvt2}.
For example, the blue boxes $\fcolorbox{blue}{white}{8}$ represent a double point closest to the special vertex. The green boxes $\fcolorbox{green}{white}{7}$ are the triple points behind the double points.
\end{Example}

\begin{Theorem}\label{thm:2}
Let $v_1,\ldots,v_n\in\mathbb{Z}^2$, $\mathcal{A}$ be the corresponding toric arrangement with special vertex.
Write $v_i=(a_i,b_i)$ and assume that $\gcd(a_i,b_i)=1$ for all $i=1,\ldots,n$.
If $\mathcal{A}$ is simplicial, then $n\le 12$.
\end{Theorem}
\begin{proof}
We use the notation from Lemma \ref{lem:clvt2}.
Let $m:=\max\{\mu_1,\ldots,\mu_n\}$ be the global maximum among the intersection points around the origin skipping the double points. Choose an $i$ such that $m$ appears in row $i$, say $m=|c_{i,j}|=|c_{i,k}|$ for $j=\sigma_i$, $k=\tau_i$, without loss of generality $i<j<k$.

By the $\gcd$-condition (Lemma \ref{lem:gcd}), $m=\gcd(c_{i,j},c_{i,k})$ divides $c_{j,k}$.
Thus $|c_{j,k}|\ge m$; by Lemma \ref{lem:clvt2} this implies $j\in[\sigma_k,\tau_k]$ and $k\in[\sigma_j,\tau_j]$,
even if $|c_{j,k}|>m$.

By the maximality of $m$, we have $\sigma_j\le k\le \tau_j$.
But $\tau_i=k$, so Lemma \ref{lem:clvt2} gives $j-i\le 4$.

Since $|c_{j,i}|=m$ we also have
$\sigma_j\le k<i+n\le \tau_j$. Again, $|c_{k,i}|=m$, so $\sigma_k\le i+n\le \tau_k$ and $k-j\le 4$ by Lemma \ref{lem:clvt2}.

Finally, $|c_{k,j}|\ge m$ gives $\sigma_k\le j\le \tau_k$ and $i+n-k\le 4$ ($j=\sigma_{i+n}$).

Adding the three inequalities gives
$$n=(j-i)+(k-j)+(i+n-k)\le 4+4+4\le 12.$$
\end{proof}

\bibliographystyle{amsalpha}

\begin{thebibliography}{CMW17}

\bibitem[BT72]{BT72}
F.~Bruhat and J.~Tits, \emph{Groupes r{\'e}ductifs sur un corps local}, Publ.
  Math., Inst. Hautes {\'E}tud. Sci. \textbf{41} (1972), 5--251 (French).

\bibitem[CH15]{CH10}
M.~Cuntz and I.~Heckenberger, \emph{Finite {W}eyl groupoids}, J. Reine Angew.
  Math. \textbf{702} (2015), 77--108.

\bibitem[CHJ20]{CHJ20}
M.~Cuntz, T.~Holm, and P.~J{\o}rgensen, \emph{Frieze patterns with
  coefficients}, Forum Math. Sigma \textbf{8} (2020), 36, Id/No e17.

\bibitem[CMW17]{CMW17}
M.~Cuntz, B.~M\"uhlherr, and Ch.~J. Weigel, \emph{Simplicial arrangements on
  convex cones}, Rend. Semin. Mat. Univ. Padova \textbf{138} (2017), 147--191.

\bibitem[Cun11]{C10}
M.~Cuntz, \emph{Crystallographic arrangements: Weyl groupoids and simplicial
  arrangements}, Bull. London Math. Soc. \textbf{43} (2011), no.~4, 734--744.

\bibitem[Cun14]{C14}
M.~Cuntz, \emph{Frieze patterns as root posets and affine triangulations}, Eur.
  J. Comb. \textbf{42} (2014), 167--178.

\bibitem[Cun23]{C23}
M.~Cuntz, \emph{Grassmannians over rings and subpolygons}, Int. Math. Res. Not.
  \textbf{2023} (2023), no.~9, 8078--8099.

\bibitem[Del72]{D72}
P.~Deligne, \emph{Les immeubles des groupes de tresses g\'en\'eralis\'es},
  Invent. Math. \textbf{17} (1972), 273--302.

\bibitem[Gr{\"u}09]{G09}
B.~Gr{\"u}nbaum, \emph{A catalogue of simplicial arrangements in the real
  projective plane}, Ars Math.~Contemp. \textbf{2} (2009), no.~1, 25 pp.

\bibitem[Mel41]{M41}
E.~Melchior, \emph{\"{U}ber {V}ielseite der projektiven {E}bene}, Deutsche
  Math. \textbf{5} (1941), 461--475.

\bibitem[PS21]{PS21}
G.~Paolini and M.~Salvetti, \emph{Proof of the {{\(K(\pi,1)\)}} conjecture for
  affine {Artin} groups}, Invent. Math. \textbf{224} (2021), no.~2, 487--572
  (English).

\end{thebibliography}

\end{document}